\documentclass[11pt,a4paper]{article}
\usepackage{amsmath,amsthm,color,eepic,xypic}
\usepackage{enumerate}
\usepackage{amssymb} 
\usepackage{verbatim}
\usepackage{hyperref}

\theoremstyle{theorem}
\newtheorem{theorem}{Theorem}
\newtheorem{proposition}{Proposition}

\newtheorem{corollary}{Corollary}

\newtheorem{example}{Example}
\newtheorem{question}{Question}

\theoremstyle{definition}

\newtheorem{claim}{Claim}

\newtheorem{remark}{Remark}

\def\F{\mathcal{F}}

\begin{document}

\title{\bf Fractional cocoloring of graphs\thanks{Gimbel and K\"undgen were supported
by the ERC Advanced Grant GRACOL, project no. 320812.}}

\date{}

\author{John Gimbel, Andr\'e K\"undgen, Michael Molloy 
\\[2mm]
} \maketitle

\begin{center}
\begin{minipage}{140mm}
\small\noindent{\bf Abstract:} 
The cochromatic number $Z(G)$ of a graph $G$ is the fewest number of colors needed to color the vertices of $G$ so that each color class is a clique or an independent set. 
In a fractional cocoloring of $G$  a non-negative weight is assigned to each clique and independent set so that for each vertex $v$, the sum of the weights of all cliques and independent sets containing $v$ is at least one. The smallest total weight of such a fractional cocoloring of $G$ is the fractional cochromatic number $Z_f(G)$. 

In this paper we prove results for the fractional cochromatic number $Z_f(G)$ that parallel results for $Z(G)$ and the well studied fractional chromatic number $\chi_f{(G)}$. For example $Z_f(G)=\chi_f(G)$ when $G$ is triangle-free, except when the only nontrivial component of $G$ is a star. More generally, if $G$ contains no $k$-clique, then $Z_f(G)\le \chi_f(G)\le Z_f(G)+R(k,k)$.
Moreover, every graph $G$ with $\chi_f(G)=m$ contains a subgraph $H$ with $Z_f(H)\ge (\frac 14 - o(1))\frac m{\log_2 m}$. 
We also prove that the maximum value of $Z_f(G)$ over all graphs $G$ of order $n$ is $\Theta (n/\log n)$, and the maximum over all graphs embedded on an orientable surface of genus $g$ is $\Theta(\sqrt g / \log g)$.
\smallskip
\noindent{\bf Keywords: Fractional coloring, cocoloring} 
\end{minipage}
\end{center}
\section{Introduction}

The fractional chromatic number $\chi_f(G)$ of a graph $G$ is {a} natural relaxation of the chromatic number $\chi(G)$ with the advantage that $\chi_f(G)$ is a lower bound for $\chi(G)$ that can be investigated by linear programming methods. The book of Scheinerman and Ullman~\cite{SU} gives a good introduction to fractional coloring and other fractional generalizations in graph theory. In this paper we initiate the study of the fractional generalization $Z_f(G)$ of the cochromatic number $Z(G)$ of a graph $G$.

The {\em cochromatic number} $Z(G)$ of a graph $G$ is the minimum number of independent sets and cliques needed to cover the vertices of $G$. This parameter was introduced in 1977 by Lesniak and Straight~\cite{LS} as a way of generalizing split graphs. Since then, many interesting papers have investigated cocoloring problems~\cite{AE, AKS, BB, CKSS, Chu, CS, EG, EGK, EGS, Gim, HKLRS, Jor, Ou, S1, S2}. 
Let $G^c$ denote the complement of a graph $G$. Clearly, $Z(G^c)=Z(G)$ and $Z(G)\le\min\{\chi(G),\chi(G^c)\}$. It follows that for large complete graphs, $Z$ and $\chi$ are very far apart. The next two results show that for certain families of graphs, they are equal. The first was originally presented in Lesniak and Straight~\cite{LS}.

\begin{theorem}\label{Thm A} 
If $G$ is a triangle-free graph and $G\neq K_2$, then $Z(G)=\chi(G)$.
\end{theorem}

Let $kG$ denote the union of $k$ disjoint copies of a graph $G$ and $\omega(G)$ denote the maximum size of a clique in $G$. A second remark of relevance was
shown in~\cite{EGS}.

\begin{theorem}\label{Thm B} 
If $k\ge \omega(G)$ then $Z(kG)=\chi(kG)=\chi(G)$.
\end{theorem}

Theorem~\ref{Thm B} shows that for families of graphs that are closed under disjoint union (such as planar graphs, graphs of bounded maximum degree, $d$-degenerate graphs, graphs of restricted girth, or any other family of graphs with a forbidden connected subgraph characterization) the maximum value of $Z$ achieved on the family is the same as that for $\chi$. For example $4K_4$ shows that the maximum value $Z(G)$ can achieve for planar graphs is 4.
 
The proofs of Theorems~\ref{Thm A} and \ref{Thm B} are similar to each other and not difficult if we consider
cocolorings of minimum order which use the fewest number of cliques in the cover. In the next section we will show that the fractional extensions of these two results hold with a few exceptions.

Erd{\H o}s, Gimbel and Straight~\cite{EGS} showed that $\chi(G)\le Z(G)+4^{\omega(G)+1}$.
In the last section we prove an extension of this to fractional cocoloring.

\section{Fractional coloring and fractional cocoloring}

A {\em fractional coloring} of a graph $G$ is an assignment of nonnegative real numbers to each independent set in $G$ so that for any fixed vertex $v$, the assignments on all independent sets containing $v$ sum to at least one.
Similarly, a {\em fractional cocoloring} of $G$ is an assignment of nonnegative real numbers to each
clique and independent set in $G$ so that for any fixed vertex $v$, the assignments on all cliques
and independent sets containing $v$ sum to at least one. 
The {\em weight} of a fractional coloring (or cocoloring) is the total sum of all the assigned numbers. The {\em fractional
cochromatic number} $Z_f(G)$ (respectively {\em fractional chromatic number} $\chi_f(G)$) of a graph $G$, is the minimum such sum taken over all fractional cocolorings (respectively fractional colorings).

The Strong Duality Principle originated by von Neuman and first published with a rigorous proof by Gale, Kuhn and Tucker~\cite{GKT}, gives alternative ways to
conceptualize both fractional colorings and fractional cocolorings. A {\em labeling} is an
assignment of nonnegative numbers to each vertex. A {\em color labeling} is an assignment
where the labels on each independent set sum to at most one. A {\em cocolor labeling} is an
assignment where the sum across each clique and independent set is at most one.
Given $g:V(G)\to[0,\infty)$ let us say the weight of $g$, denoted by $w(g)$, is the sum $\sum_{v\in V(G)}g(v)$. It is easy to see that for any cocolor labeling $g$ of a graph $G$, $w(g)\le Z_f(G)$. In fact, by
the Strong Duality Principle, the fractional cochromatic number of $G$ equals the maximum weight of a cocolor labeling of $G$ (just like the fractional chromatic number of $G$ equals the maximum weight of a color labeling of $G$). 

This connection makes it easy to verify $Z_f(G)$ for specific graphs $G$.

\begin{example}\label{Exm C} 
If  $G=K_{1,t}$ with $t\ge 1$, then $Z_f(G)=2-\frac 1t$. The lower bound follows from the cocolor labeling in which the center of the star receives $1-\frac 1t$ and each leaf $\frac 1t$. The upper bound follows from a cover where each edge receives weight $\frac 1t$, and the set of all leaves weight $1-\frac 1t$.

If $G=K_{1,t}\cup K_s^c$, the $t$-star along with $s$ isolated vertices, where $s,t\ge 1$, then $Z_f(G)=2-\frac1{t+1}$. The lower bound follows from the cocolor labeling in which the center receives $1-\frac 1{t+1}$ and each leaf as well as one other vertex $\frac 1{t+1}$. The upper bound follows from a cover where each edge receives weight $\frac1{t+1}$, and the two maximal independent sets receive weights $1-\frac 1{t+1}$ and $\frac 1{t+1}$.
\end{example}

General results for the fractional chromatic number whose proof is based on duality carry over similarly.
For example if $\alpha(G)$ denotes the size of a maximum independent set in a graph $G$ on $n$ vertices, then it is easy to see that $\chi_f(G)\ge n/\alpha(G)$. Moreover, equality holds if $G$ is vertex transitive (see page 42 in~\cite{SU}.) Correspondingly, we obtain   

\begin{proposition}\label{trivial Zf}
If $G$ is a graph on $n$ vertices with $k=\max\{\alpha(G),\omega(G)\}$, then $Z_f(G)\ge n/k$.
Equality holds, for example, when $G$ is vertex-transitive.
\end{proposition}

\begin{proof}
The inequality follows from the cocolor labeling in which every vertex has weight $1/k$.
Now suppose that $G$ is vertex-transitive, and suppose that $k=\alpha(G)$. Thus every vertex lies in the same number of independent sets of size $k$; let $\ell$ denote this number. Equality is obtained by assigning a weight of $1/\ell$ to every  independent set of size $k$. The case
$k=\omega(G)$ follows from a nearly identical argument, or by applying the previous case to the complement of $G$.
\end{proof}

Strong duality also yields a fractional counterpart to Theorem~\ref{Thm B}.

\begin{theorem}\label{Thm E} 
Suppose $k$ is at least the clique number of $G$. 
Then $Z_f(kG)=\chi_f(kG)=\chi_f(G).$
\end{theorem}

\begin{proof} The second equality is trivial. Further{more}, $Z_f(kG)\le\chi_f(kG)$. So let $f :V(kG)\to[0,\infty)$ 
be a color label with weight $\chi_f(kG)$. We create a new labeling $f'$ of $kG$ as follows: For each vertex $v$ in $kG$, let $f'(v)$ be the average of the label $f$ on all copies
of $v$ in $kG$. We note that summing $f'$ across all copies of $v$ yields the same value as
summing $f$ across all copies of $v$. Thus $w(f')=w(f)=\chi_f(G)$.

We claim that $f'$ is a color labeling. For suppose there exists an independent set of
vertices with the property that summing $f'$ across them yields a value greater than one.
Then in some copy of $G$ there must be a collection, say $S$, of independent vertices where $f'$, summed
over $S$, is greater than $1/k$. Let $S'$ be $S$ together with all copies of vertices from $S$ in
other copies of $G$. Note, $S'$ is an independent set in $kG$, but when we sum $f'$
across $S'$ we get a number larger than one. But summing $f$ over $S'$ we get the same sum. 
This contradiction implies that $f'$ is a color labeling of $kG$.

Next, we claim that $f'$ is in fact a cocolor labeling of $kG$. For suppose there is some clique in
$kG$ where $f'$ sums across the clique to a value larger than one. Then some vertex in the
clique must have a label greater than $1/k$. Taking all copies of this vertex produces an
independent set, say $J$, with the property that summing $f'$ across $J$ yields a value larger
than one, a contradiction. 

Thus, $f'$ is a cocolor labeling of $kG$ with $w(f')=\chi_f(kG)$, so that $Z_f(kG)\geq \chi_f(kG)$.
\end{proof}

Gr\"otschel, Lov\'asz and Schrijver~\cite{GLS} observed that for each fixed rational number $r>2$ determining whether a graph $G$ has $\chi_f(G)\le r$ is NP-complete
(see also~\cite{SU}, Theorem 3.9.2.) As the transformation from the $n$-vertex graph $G$ to $nG$ can be done in polynomial time Theorem~\ref{Thm E} implies the following.

\begin{corollary}\label{Cor E}
 For each fixed rational number $r>2$, the problem of deciding if a graph $G$ has $Z_f(G)\le r$ is NP-complete. 
\end{corollary}

Example~\ref{Exm C} shows that Theorem~\ref{Thm A} can't be immediately extended to fractional
cocolorings, since for a bipartite graph $G$ we must have $\chi_f(G)=\chi(G)=2$, as long as it has edges.  
As our next result shows, in extending Theorem~\ref{Thm A}, these are our only troublesome cases.

\begin{theorem}\label{Thm D} 
If $G$ is a triangle-free graph not found in Example~\ref{Exm C}, then $\chi_f(G)=Z_f(G)$.
\end{theorem}

\begin{proof}
Let $G$ be a triangle-free graph and consider a fractional cocoloring of minimum size, such that the total weight on $K_2$'s is minimal. If we can show that the weight on each $K_2$ must be zero, then it follows that $\chi_f(G)=Z_f(G)$.

\begin{claim}\label{ee}
There are no disjoint edges of positive weight.  
\end{claim}

Suppose $ab$ and $cd$ are disjoint edges with $w(ab)=\alpha$, $w(cd)=\beta$, where $\alpha\ge \beta>0$.
If there is an edge from $\{a,b\}$ to $\{c,d\}$, then we may assume that this edge is $ad$. Since $G$ is triangle-free it follows that we can always assume that $\{a,c\}$ and $\{b,d\}$ are independent sets. Reducing the weight of $ab$ and $cd$ by $\beta$ and adding $\beta$ to the weight of these independent sets we get a cocoloring that contradicts minimality. 

\begin{claim}\label{Ie}
Every edge of positive weight intersects every independent set of positive weight.
\end{claim}

Suppose $I$ is an independent set of weight $\alpha>0$ that is disjoint from the edge $ab$ of weight $\beta>0$.
Since $G$ is triangle-free we can partition $I$ into two sets $A,B$ such that $A'=A\cup\{a\}$ and $B'=B\cup\{b\}$ are independent sets.
Reducing the weights of $I$ and $ab$ by the smaller of $\alpha,\beta$ and increasing the weight of both $A',B'$ by the same amount, we again get a cocoloring that contradicts minimality.

\begin{claim}\label{noboth}
An edge of positive weight is not incident with edges at both of its endpoints.
\end{claim}

Suppose $w(ab)>0$ and we have edges $ac$ and $bd$. Since $G$ is triangle-free we have $c\neq  d$ and so by Claim~\ref{ee} these edges cannot both have positive weight, so say $w(ac)=0$. Now observe that there can be no edge incident with $c$ that has positive weight, since otherwise either we get a triangle, or a contradiction to Claim~\ref{ee}. Furthermore observe that by Claim~\ref{Ie} every independent set of positive weight containing $c$ must intersect $ab$, and thus contain $b$. Since the independent sets containing $c$ have total weight at least 1, the same can be said for $b$. Thus we can move all the positive weight from $ab$ to the independent set $\{a\}$, contradicting minimality. 

\begin{claim}\label{star}
If two edges  are incident then they  have the same weight.
\end{claim}

Suppose to the contrary that $vu,vw$ are incident edges with weights $\alpha,\beta$, resp.\ where $\alpha>\beta$ (and possibly $\beta=0$).  There is no other edge of positive weight incident with $w$ by Claim~\ref{ee} (for $uv$) and the fact that $G$ is triangle-free.
Moreover, every independent set of positive weight that contains $w$ may not contain $v$, and thus must contain $u$ by Claim~\ref{Ie}. Hence the total weight on the independent sets containing $w$ (and thus $u$) must be at least $1-\beta$ and we can move a weight of $\alpha-\beta$ from $uv$ to $\{v\}$, contradicting minimality. 

So it follows from Claims~\ref{ee} and~\ref{noboth}  that the edges of positive weight form a component $K_{1,t}$ of $G$ and from Claim~\ref{star} that every such edge has  the same weight $\alpha>0$. Let $v$ be the center of this star $K_{1,t}$.
  
It remains to show that there is no edge $ab$ in $G-K_{1,t}$. If there was such an edge then the vertices $a,b$ can only be covered by independent sets, and the independent sets containing $a$ must be distinct from those of $b$ and have total weight at least 1. For each independent set $I$ containing $a$: (i) consider the independent set $I'$ formed by removing the vertices of $K_{1,t}$ from $I$ and then adding $v$; and (ii) move all the weight from $I$ to $I'$. (If $I=I'$ then this means that the weight remains on $I$.)
For each independent set $S$ containing $b$: (i) consider the independent set $S'$ formed by removing the vertices of $K_{1,t}$ from $S$ and then adding $N(v)$; and (ii) move all the weight from $S$ to $S'$. The vertices of $K_{1,t}$ are now covered by  these independent sets,  and so we can remove all weight from the edges of $K_{1,t}$, thus  obtaining a contradiction to minimality. Therefore we know that $G=K_{1,t}+(n-t-1)K_1$ for $t\ge 1$. 
\end{proof}

\begin{remark}
Larsen, Propp and Ullman~\cite{LPU} proved that applying Mycielski's construction to a graph with fractional chromatic number $c$ yields a graph with fractional chromatic number $c+\frac 1c$. Fisher~\cite{Fi} used this to show that there are families of graphs $G$, such that the denominator of $\chi_f(G)$ grows exponentially in $|V(G)|$, whereas Chv\'atal, Garey and Johnson~\cite{CGJ} had shown that in general the denominator of $\chi_f(G)$ for an $n$-vertex graph $G$ can be no larger than $n^{n/2}$. Since $C_5$ is triangle-free and the Mycielskian of a triangle-free graph is triangle-free, the results of Fisher combine with Theorem~\ref{Thm D} to show that the denominator of $Z_f(G)$ can also grow exponentially in $|V(G)|$ 
\end{remark}

\begin{example}
The Kneser Graph $K_{a:b}$ has as vertex set all $b$-element subsets of an $a$-element set, where two vertices are adjacent if the sets are disjoint. Lov\'asz~\cite{Lov78} showed that $\chi(K_{a:b})=a-2b+2$.
Furthermore, it is not hard to show that $\chi_f(K_{a:b})=a/b$, using the fact that Kneser graphs are vertex-transitive (see~\cite{SU} page 45). 
Hence, $\chi(K_{3k-1:k})=k+1$ and $\chi_f(K_{3k-1:k})=3-1/k$. Since $K_{3k-1:k}$ is triangle-free this yields
graphs with fractional chromatic and (by Theorem~\ref{Thm D}) fractional cochromatic numbers close to three
having arbitrarily large chromatic and (by Theorem~\ref{Thm A}) cochromatic numbers.
\end{example}


\section{Probabilistic results}

Let $z(n)$ be the largest cochromatic number among all graphs with $n$ vertices, and let $G_{n,p}$ denote the random graph on $n$ labelled vertices with edge probability $p$. Erd\H os, Gimbel and Kratsch~\cite{EGK} proved that for sufficiently large $n$, $\frac n{2\log_2 n}<z(n)<(2+o(1))\frac n{\log_2 n}$, where the lower bound is given by $G_{n,1/2}$. (They also show that given an infinite family of graphs with cochromatic number $z$, there is a $c>0$ such that these graphs all have at least $cz^2\log_2^2(z)$ edges.) 
Remark~\ref{Rem 6} below shows that this result doesn't change if we consider fractional cocoloring instead of cocoloring.

Bollob\'as~\cite{Bol} showed that asymptotically almost surely (a.a.s.) $\chi(G_{n,1/2})\approx \frac n{2\log_2 n}$, and Matula~\cite{Mat} showed that a.a.s. $\alpha(G_{n,1/2})\approx 2 \log_2 n$ and $\omega(G_{n,1/2})\approx 2\log_2 n$.
Thus Proposition~\ref{trivial Zf} and $Z_f(G)\le \chi(G)$ imply the following.

\begin{remark}\label{Rem 6} 
The random graph $G_{n,1/2}$ asymptotically almost surely satisfies $Z_f(G_{n,1/2})\approx \frac n{2\log_2 n}$.
\end{remark}

As noted chromatic and cochromatic number can be very far apart. Considering
complete graphs, the fractional versions can also be far apart. Furthermore, every induced
subgraph of a clique has fractional cochromatic number equal to one. Dropping the
notion of induced, Theorem~\ref{subgraph thm} below shows that if a graph has large fractional chromatic number, it has a
subgraph with large fractional cochromatic number.

Alon, Krivelevich and Sudakov~\cite{AKS} settled a problem of Erd\"os and Gimbel, by proving that every graph of chromatic number $n$ contains a subgraph with cochromatic number at least $\Omega( n/ \log n)$. (The complete graph together with the value of $z(n)$ shows that this is tight, up to the constant factor.) We will prove a fractional analogue of this result using a very similar proof. 

\begin{theorem}\label{subgraph thm}
If $G$ is a graph with $\chi_f(G)=m$, then $G$ has a subgraph $H$ with $Z_f(H)\ge (\frac 14-o(1))\frac m{\log_2 m}$.
Furthermore, $G=K_m$ shows that this is tight up to a constant factor.
\end{theorem}

\begin{proof}
If the fractional cochromatic number of $G$ is greater than $\frac 12 m/\log_2 m$, then we are done (set $H=G$).  So let $Z$ be a fractional cocoloring of $G$ of total weight at most $\frac 12 m/\log_2 m$. Let $V_1$ be the set of vertices $v$ such that $Z$ assigns a total weight of at least $1/2$ to the cliques containing $v$, and let $G_1$ be the subgraph of $G$ induced by $V_1$.

Every clique in $G$ has size at most $m$, as $\chi_f(G)=m$. Thus the weights of all the cliques used in $Z$ sums to at  least $\frac 12 |V_1|/m$, since each clique contributes to the weight of at most $m$ vertices of $V_1$. It follows that $\frac 12 |V_1|/m \le Z_f(G)\le \frac 12m/\log_2 m$, and thus $|V_1| < m^2$.
Doubling the $Z$-weights on all the independent sets will yield a fractional coloring that covers all vertices not in $V_1$, so  $\chi_f(G-V_1)\le 2Z_f(G)\leq  m/2\log_2 m$. Therefore
\[\chi_f(G_1)\geq\chi_f(G)-\chi_f(G-V_1)=m-o(m).\]

We finish the argument by picking a random subgraph $H$ of $G_1$ (rather than $G$) where each edge is chosen with probability $1/2$. The exact same argument as in the proof of Lemma 2.2 in~\cite{AKS} now shows that with probability approaching 1, the vertex-set of every clique or independent set in $H$ induces a subgraph of $G_1$ that has chromatic number at most $4 \log_2 m$ in $G_1$. Therefore, any fractional cocoloring of $H$ can be converted to a fractional coloring of $G_1$, where the total weight is increased at most by a factor of $4\log_2 m$.
So $H$ does not have a fractional cocoloring of total weight less than $\frac{m-o(m)}{4\log_2 m}$.
\end{proof}


\section{Fractional cocoloring on surfaces}

For a given surface $\Sigma$ and a graph parameter $f$, let $f(\Sigma)$ denote the maximum value $f(G)$ can obtain for any graph $G$ that is embedded on $\Sigma$.
Heawood~\cite{Hea} proved that if $S_g$ is an orientable surface with genus $g>0$, then $\chi(S_g)\le \frac 12(7+\sqrt{1+48g})=H(g)$.
Ringel and Youngs~\cite{RY} showed that in fact equality must hold, by finding embeddings of $K_{H(g)}$ on $S_g$. This combines with the 4 color theorem to show that $\chi_f(S_g)=\chi(S_g)=H(g)$ for all $g\ge 0$.

Straight~\cite{S1,S2} conjectured, and verified for small genus, that  $Z(S_g)$ is the maximum $n$ such that $K_1\cup K_2\cup\dots \cup K_n$ can be embedded on $S_g$. 
Gimbel~\cite{Gim} disproved this nondeterministically, and with Thomassen~\cite{GT} extended this to prove that $Z(S_g)=\Theta(\frac{\sqrt g}{\log g})$.
However, the smallest genus for which Straight's conjecture is false is not known. For the fractional cochromatic number we similarly obtain

\begin{corollary}
$Z_f(S_g)=\Theta(\frac{\sqrt g}{\log g})$.
\end{corollary}

\begin{proof}
Ringel and Youngs~\cite{RY} proved that $G=K_{H(g)}$ embeds on $S_g$, and thus $\chi_f(G)=\Theta(\sqrt g)$. Therefore Theorem~\ref{subgraph thm} implies that $G$ contains a subgraph $H$ with $Z_f(H)\ge \Omega\left(\frac{\sqrt g}{\log g}\right)$. The upper bound follows, since $Z_f(S_g)\le Z(S_g)=\Theta(\frac{\sqrt g}{\log g})$.
\end{proof}


\section{Another bound}

Erd{\H o}s, Gimbel and Straight~\cite{EGS} showed that a graph with clique number $k-1$ satisfies $\chi(G)\le Z(G)+R(k,k)$, where $R(k,k)$ denotes the ordinary Ramsey number. In this section we prove an extension of this to fractional cocoloring.

\begin{theorem}\label{Ramsey bound}
For every graph $G$ with $k=\omega(G)+1$, $\chi_f(G)\le Z_f(G)+R(k,k)$.

\end{theorem}

\begin{proof}
Let $R=R(k,k)$ and $n=|V(G)|$. Consider a covering achieving $Z_f(G)$ and let $Z_\alpha$ be the total weight on the independent sets in this cover, and $Z_\omega$ the total weight on the cliques. For $1\le i\le n$ we will let $V_i=\{v\in V(G): \frac{i-1}n<$ total weight on cliques containing $v \le \frac in\}$. Thus $Z_\omega \ge \sum_{i=1}^n\frac{i-1}n \frac{|V_i|}{k}$.  

We will now remove all the cliques from the cover and for each $1\le i\le n$ replace them with $s_i$ independent sets of size $k$ with weight $\frac in$  as follows. If $|V_1|<R(k,k)=R$, then we let $s_1=0$, but otherwise (since $G$ has no cliques of size $k$) we can remove some $s_1$ independent sets of size $k$ from $V_1$ until we are left with less than $R$ vertices; we denote this number by $R_1=|V_1|-s_1k<R$. Now we include these $R_1$ vertices from $V_1$ in $V_2$ and remove some $s_2$ independent sets of size $k$ from this new set, until we obtain $R_2=|V_2|+R_1-s_2k<R$ remaining vertices. We continue in this manner to obtain $s_i$ independent sets of size $k$ that cover all but $R_i$ remaining vertices in the set obtained from $V_i$ by adding all uncovered vertices from $V_1,\dots,V_{i-1}$, where $R_i=|V_i|+R_{i-1}-s_ik<R$ (and we let $R_0=0$ for convenience).

We now obtain a new cover by giving the $s_i=\frac 1k(|V_i|+R_{i-1}-R_i)$ independent sets of size $k$ a weight of $i/n$, and covering the remaining $R_n<R$ vertices with independent sets of size 1.  By the definition of $V_i$, every $v\in V_i$ was covered by cliques of total weight at most $\frac{i}{n}$ and these are now replaced by independent sets of total weight at least $\frac{i}{n}$.  So we  have a fractional coloring, and thus
\begin{eqnarray*}
\chi_f(G) &\le& Z_\alpha+R_n+\sum_{i=1}^n\frac in \cdot \frac{|V_i|+R_{i-1}-R_i}k\le Z_\alpha+R_n+Z_\omega+\sum_{i=1}^n \frac{|V_i|+R_{i-1}}{nk}-\frac{R_n}k \\
& \le& Z_f(G)+R_n+\frac1k+\frac {R-1}k-\frac {R_n}k\le Z_f(G)+R.
\end{eqnarray*}
\end{proof}

\begin{remark}\label{rem:Xf-Zf}

Since 
$2^{k/2}\le R(k,k)\le 4^k$, this raises the question about the necessary size of the error term.  
Building on the idea of Theorem 4 in~\cite{EGS}, Erd{\H o}s and Gimbel~\cite{EG} show that in fact for every $\varepsilon >0$ and $k$ sufficiently large there is a graph $G$ with $\omega(G)<k$ and $\chi(G)\ge Z(G)+(2-2\varepsilon)^{k/2}$. The same graphs satisfy
$\chi_f(G)\ge Z_f(G)+(2-2\varepsilon)^{k/2}$ as well: For $p=\frac1{2-\varepsilon}$, it is shown that the random graph $G=G_{n,p}$ on $n=(2-2\varepsilon)^{k/2}$ vertices satisfies $\chi(G)-Z(G)\ge |V(G)|/\alpha(G)-\chi(G^c)\geq (2-2\varepsilon)^{k/2}$. But $\chi_f(G)-Z_f(G)\ge |V(G)|/\alpha(G)-\chi(G^c)$ is also valid.
\end{remark}


\section{Open questions}

Let $Z_f(n,m)$ be the maximum value $Z_f(G)$ can take over all graphs $G$ on $n$ vertices and $m$ edges ,and let 
Since $Z_f(G^c)=Z_f(G)$ it follows that $Z_f(n,m)$ is symmetric, in that $Z_f(n,m)=Z_f(n,\binom n2 -m)$.
$Z_f(n,m)$ achieves a minimum of 1 when $m=0$ or $\binom n2$, but the situation for the maximum is less obvious:

\begin{question}
Given $n$, for which $m$ is $Z_f(n,m)= Z_f(n)$?
\end{question}

 Remark~\ref{Rem 6} seems to suggest that $m\approx\frac 12\binom n2$.
In this light it makes sense to ask if $Z_f(n,m)$ is unimodal in $m$: 

\begin{question}
Is it true that if $m<\frac 12\binom n2$, then $Z_f(n,m)\le Z_f(n,m+1)$?
\end{question}

Determining the (fractional) cochromatic number of a graph is NP-hard. In~\cite{HKLRS} it is shown that for fixed $k$ the question if a perfect graph $G$ has $Z(G)\le k$ can be decided in time $O(n\log n)$. Is there a similar algorithm for fractional cocoloring? For graphs of bounded tree-width it is shown in~\cite{CKSS} that the cochromatic number can be found in polynomial time. Does this also extend to fractional cocoloring?

What can we say about the structure of critically fractionally cochromatic graphs -- that is, graphs with the property that the removal of any vertex necessarily reduces the fractional cochromatic number? See also~\cite{BB,Jor,Ou}.

In Theorem~\ref{Thm D} it is shown that $\chi_f(G) - Z_f(G)\le 1$ when $G$ is triangle-free.
How large can this difference be when $G$ is $K_4$-free, or more generally $K_k$-free for some fixed $k$?
As described above (see Remark~\ref{rem:Xf-Zf}), the graphs from~\cite{EG} show that the difference can grow exponentially in $k$ both for the fractional and integral versions.   
We don't know how much bigger it can be for $\chi_f(G)-Z_f(G)$ or for $\chi(G)-Z(G)$.
Theorem~\ref{Ramsey bound} and~\cite{EGS} show that these differences grow at most exponentially in $k$. More detailed information could be challenging to obtain.

Given a finite set of graphs $\F$ we say that a graph is $\F$-free if it does not have an induced subgraph isomorphic to a member of $\F$.
Gy\'arf\'as~\cite{Gya} and Sumner~\cite{Sum} independently conjectured that for any fixed tree $T$ and any fixed integer $k$, the
family of $\{T, K_k\}$-free graphs has bounded chromatic number. (This is best possible, since Erd\"os proved that there are graphs $G$ of arbitrarily large girth and large ratio $|V(G)|/\alpha(G)$.) 
Seymour and Chudnovsky~\cite{CS, Chu} prove that the Gy\'arf\'as-Sumner conjecture is equivalent to showing that the family of $\F$-free graphs has bounded cochromatic number if and only if $\F$ contains a forest, a complement of a forest, a complete multipartite graph and the complement of a complete multipartite graph. It would be interesting to study fractional versions of these two conjectures.

\bibliographystyle{siam}
\bibliography{fractionalCocoloring}

\end{document}